\documentclass[a4paper,11pt]{amsart}
\usepackage[utf8]{inputenc}

\usepackage{amssymb,amsfonts,amsmath}
\usepackage[english]{babel}
\usepackage{a4wide}
\usepackage{graphicx}
\usepackage{dsfont}
\usepackage{mathrsfs} 
\usepackage{mathtools} 
\usepackage{amsthm}
\usepackage{xcolor}
\usepackage{hyperref}

\usepackage{url}

\newcommand{\NN}{\mathbb N}
\newcommand{\FF}{\mathbb F}
\newcommand{\PP}{\mathbb P}

\newcommand{\ZZ}{\mathbb Z}
\newcommand{\RR}{\mathbb R}
\newcommand{\Scal}{\mathcal{S}}
\newcommand{\bsx}{\boldsymbol{x}}

\newtheorem{theorem}{Theorem}

\newtheorem{definition}{Definition}
\newtheorem{proposition}{Proposition}
\newtheorem{remark}{Remark}
\newtheorem{lemma}{Lemma}

\newtheorem{op}{Open Problem}

\newcommand{\notiz}[1]{}

\begin{document}

\title[Disproving quasi-uniformity of Halton sequences]{Disproving the quasi-uniformity of the Halton sequences and of some Halton-type sequences}
\author[T.~Goda]{Takashi Goda}
\address[T.~Goda]{Graduate School of Engineering, The University of Tokyo, 7-3-1 Hongo, Bunkyo-ku, Tokyo 113-8656,
Japan} 
\email[]{goda@frcer.t.u-tokyo.ac.jp}
\thanks{The work of T.G.\ is supported by JSPS KAKENHI Grant Number 23K03210.}

\author[R.~Hofer]{Roswitha Hofer}
\address[R.~Hofer]{Institute of Financial Mathematics and Applied Number Theory, Johannes Kepler University Linz, Altenbergerstr.69, 4040 Linz, Austria}
\email[]{roswitha.hofer@jku.at}
\thanks{}

\author[K.~Suzuki]{Kosuke Suzuki}
\address[K.~Suzuki]{Faculty of Science, Yamagata University, 1-4-12 Kojirakawa-machi, Yamagata, 990-8560, Japan}
\email[]{kosuke-suzuki@sci.kj.yamagata-u.ac.jp}
\thanks{The work of K.S.\ is supported by JSPS KAKENHI Grant Number 24K06857.}

\date{\today}


\begin{abstract}
    In this short article, we prove that the Halton sequence, one of the most well-known low-discrepancy sequences, is not quasi-uniform in any dimension $d \ge 2$ with any pairwise relatively prime bases. We further disprove the quasi-uniformity of some Halton-type sequences, including the $p$-dimensional Faure sequence in base $p$, $p\in\PP$, which provides an alternative proof of the known results.
\end{abstract}

\maketitle

%\sloppy

\section{Introduction}\label{sec:intro}

The Halton sequence is evidently one of the most well-known low-discrepancy sequences in the unit cube $[0,1]^d$ for any dimension $d\ge 1$ \cite{LP14,Nie92}. 
Throughout this paper, for an integer $b\ge 2$, we abbreviate the base-$b$ representation of $n=n_0+n_1b+n_2b^2+\cdots\in \NN_0$ as $(n_0,n_1,n_2,\ldots)_b$, where $\NN_0$ denotes the set of non-negative integers and $n_0,n_1,\ldots\in\{0,1,\ldots,b-1\}$.
The radical inverse function $\varphi_b$ in base $b$ is then defined by 
\[ \varphi_b:\NN_0\to[0,1),\,n=(n_0,n_1,n_2,\ldots)_b\mapsto \frac{n_0}{b}+\frac{n_1}{b^2}+\frac{n_2}{b^3}+\cdots. \]
For pairwise relatively prime integers $b_1,\ldots,b_d\ge 2$, the corresponding $d$-dimensional Halton sequence is defined as the sequence $\{\bsx_0,\bsx_1,\ldots\}$ with
\[ \bsx_n = \left( \varphi_{b_1}(n),\ldots,\varphi_{b_d}(n)\right)\in [0,1]^d. \]
In the case $d=1$, this construction is nothing but the renowned van der Corput sequence. 

The discrepancy of the Halton sequence was studied already by Halton himself in \cite{Hal60}, who proved an upper bound of $\mathcal{O}(N^{-1}(\log N)^d)$ for the star discrepancy of the initial $N$ points of the sequence. Much later, Atanassov \cite{Ata04} provided a refined analysis, which led to a substantial improvement of the implied constant in this $\mathcal{O}$-bound. While Halton's original constant grows super-exponentially fast in the dimension $d$, Atanassov's constant decreases to $0$ super-exponentially. In \cite{Le16a,Le16b}, Levin proved a matching-order lower bound for the star discrepancy of the Halton sequence. Moreover, the tractability property of the weighted star discrepancy of Halton sequences has also been studied in \cite{HPT19}.

With these results, together with the classical Koksma--Hlawka inequality \cite[Chapter~2]{KN74}, the use of Halton sequences in high-dimensional numerical integration is theoretically well justified. However, Halton sequences have also been employed as sampling nodes for scattered data approximation---particularly in the context of radial basis function approximation (see, for instance, \cite{FM12,MS17,SSW23}). In such applications, not the low discrepancy but rather the geometric properties of the point set, such as the coverage of the underlying domain and the spacing between points, are crucial, as they contribute to bounds on the approximation error and to numerical stability \cite{DS10,S95,W05,WSH21}. These geometric properties can be quantified in terms of the \emph{covering radius} and the \emph{separation radius} of the point set, naturally leading to the notion of quasi-uniformity.

Let $\Scal = \{ \bsx_0, \bsx_1, \ldots \}$ be an infinite sequence of points in $[0,1]^d$, and let $P_N = \{ \bsx_0, \ldots, \bsx_{N-1} \}$ denote the set of the first $N$ points of $\Scal$. The covering and separation radii of $P_N$ are then defined by
\[ h(P_N):=\sup_{\bsx\in [0,1]^d}\min_{0\le i<N}\|\bsx-\bsx_i\|, \]
and
\[ q(P_N) := \frac{1}{2}\min_{0\le i<j<N}\|\bsx_i-\bsx_j\|,\]
respectively, where $\|\cdot\|$ denotes the Euclidean norm in $\RR^d$. The sequence $\Scal$ is said to be \emph{quasi-uniform} if there exists a constant $C>0$ such that $h(P_N)/q(P_N)\leq C$ for all $N\ge 2$. Assuming that the dimension $d$ is finite, the Euclidean norm may be replaced by the $\ell_p$ norm for any $p \in [1, \infty]$ without affecting the notion of quasi-uniformity.

A simple volumetric argument shows that there exist constants $C_h,C_q>0$ such that, for any $N$-element point set $P_N$, we have
\[ h(P_N)\ge C_hN^{-1/d}\quad \text{and}\quad q(P_N)\le C_q N^{-1/d}, \]
see, for instance, \cite[Appendix~A.1]{DGLPS25}.
Thus, the quasi-uniformity requires that both radii decay at the optimal rate of order $N^{-1/d}$.
As shown, for instance, in \cite[Theorem~6.12]{Nie92}, it is known that for $\Scal$ being the Halton sequence in any relatively prime bases $b_1, \ldots, b_d\ge 2$, 
\[
h(P_N) \le \frac{\sqrt{d}}{N^{1/d}} \max_{1 \le \ell \le d} b_{\ell},\quad \text{for all $N\ge 1$},
\]
which implies that the Halton sequence achieves the optimal rate for the covering radius. Therefore, whether the Halton sequence is quasi-uniform or not is determined by whether its separation radius decays no faster than $N^{-1/d}$.

Figure~\ref{fig:halton_separation} shows the separation radius of the Halton sequences up to $N=10^5$ in dimensions $d\in \{2,3,4,5\}$. Here, the bases $b_1,\ldots,b_d$ are chosen to be the first $d$ prime numbers. The reference lines of $N^{-1/d}$ are also plotted. It suggests that the separation radius decays faster than the optimal rate, implying that the Halton sequences are \emph{not} quasi-uniform. To the best of our knowledge, however, no theoretical result is available in the literature.  

\begin{figure}[t]
    \centering
    \includegraphics[width=0.7\linewidth]{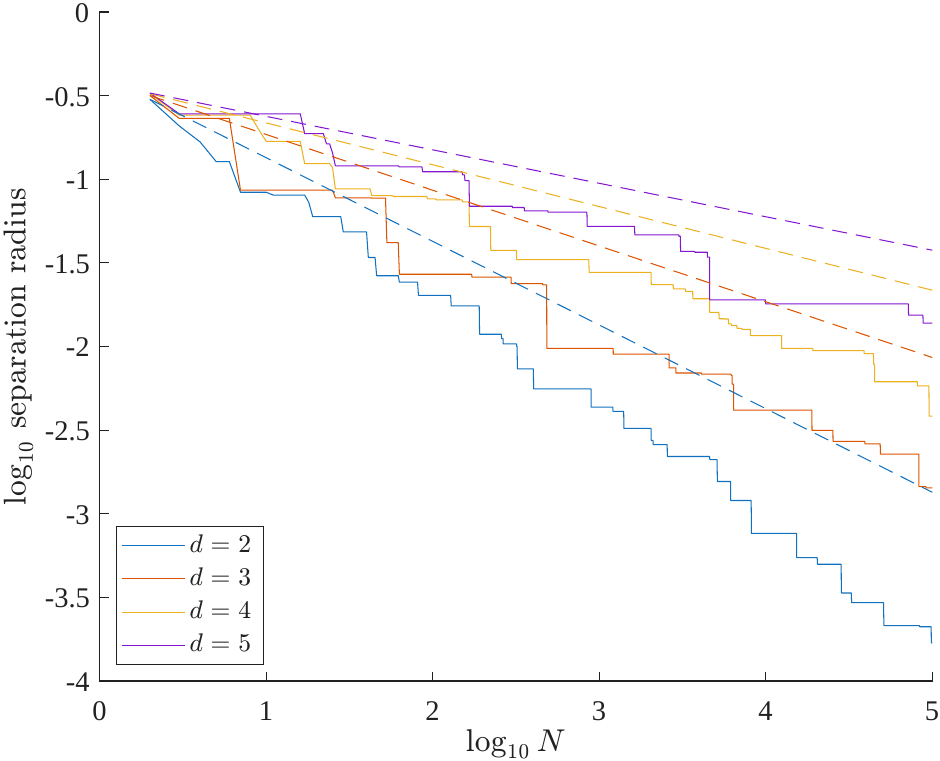}
    \caption{Separation radius of the Halton sequence from $N=2$ to $N=10^5$ in dimensions $2$ (blue), $3$ (red), $4$ (orange), and $5$ (purple). The corresponding reference lines $N^{-1/d}$ are also plotted in the same colors.}
    \label{fig:halton_separation}
\end{figure}

In passing, it is known that the first point $\bsx_0$, which lies at the origin, is well separated from the subsequent points of the Halton sequence \cite{HKZ05,Owe06}. Let $x_{i,\ell}$ denote the $\ell$-th coordinate of the $i$-th point $\bsx_i$. In the present context, this separation property can be rephrased as the existence of a constant $c_d > 0$, depending only on the bases $b_1,\ldots,b_d$, such that
\[ \min_{1\le i<N}\|\bsx_0-\bsx_i\|\ge \sqrt{d}\min_{1\le i<N}\prod_{\ell=1}^{d}x_{i,\ell}^{1/d}\ge \frac{c_d}{N^{1/d}},\quad \text{for all $N\ge 2$.}\]
Here, the first inequality follows from the arithmetic--geometric mean inequality, and the second inequality is from \cite[Theorem~3.1]{Owe06}.

One of the main purposes of this article is to disprove the quasi-uniformity of the Halton sequences in any dimension $d\ge 2$. This is done in Section~\ref{sec:Halton} by proving that the separation radius of the Halton sequence indeed decays faster than $N^{-1/d}$ for any $d\ge 2$ (Theorem~\ref{thm:halton_not_quasi-uniform}). Furthermore, in Section~\ref{sec:Halton-type}, we consider the \emph{Halton-type sequences} in the sense of Hofer \cite{Ho13,Ho18}, and show that some of those sequences are also not quasi-uniform. These include the $p$-dimensional Faure sequence in base $p$ with $p\in\PP$ \cite{Fa82}, which was already disproven to be quasi-uniform in \cite{DGS25}. Our purpose here is to provide an alternative proof from the Halton-type point of view.

\subsection{Some known results}
Before moving on, we summarize several known results about the quasi-uniformity of low-discrepancy point sets and sequences. 

The covering radius has been studied under the name of \emph{dispersion} \cite[Section~6]{Nie92}. It is known that certain classes of low-discrepancy constructions, such as $(t,d)$-sequences and the Halton sequence, achieve the optimal rate $N^{-1/d}$; see \cite[Theorems~6.11 \& 6.12]{Nie92}. 

The separation radius, however, has not been studied to the same extent, and it has therefore remained unclear whether common low-discrepancy constructions are quasi-uniform. This question was explicitly raised in \cite{WBG21}. In \cite{SS07}, Sobol' and Shukhman studied the separation radius of the Sobol' sequence \cite{So67} and, based on numerical experiments, suggested that it also attains the optimal rate $N^{-1/d}$ for the separation radius. In the case $d=2$, however, this prediction was disproven recently in \cite{G24,S25}. This shows that the Sobol' sequence is not quasi-uniform, at least in two dimensions, whereas the question remains open for $ d\ge 3$. 

Further examples of low-discrepancy point sets and sequences that are \emph{not} quasi-uniform can be found in \cite{DGS25}, including the $p$-dimensional Faure sequence in base $p$ with $p\in\PP$. Examples of low-discrepancy point sets and sequences that \emph{are} quasi-uniform are given in \cite{DGLPS25,DGS25,G24b}, in addition to the one-dimensional van der Corput sequence in any base $b$. The latter result means that any one-dimensional projection of the Halton sequence is quasi-uniform, whereas our main result below implies that no higher-dimensional projection satisfies this property.

\section{Halton sequences are not quasi-uniform}\label{sec:Halton}

Throughout this section, let $b_1,\ldots,b_d$ be pairwise relatively prime integers with $1<b_1 < \cdots < b_d$. We denote Euler's totient function of $b_1$ by $r$, i.e. $r=\phi(b_1)$.  
Note that then $b_j^r \equiv 1 \pmod{b_1}$ holds for $2 \le j \le d$ by Euler's Theorem. In this section, we prove the following result.

\begin{theorem}\label{thm:halton_not_quasi-uniform}
Let $d \ge 2$, and let $b_1, \dots, b_d$ be pairwise relatively prime integers with $1<b_1 < \cdots < b_d$.
Let $\Scal$ be the $d$-dimensional Halton sequence in bases $b_1,\dots,b_d$.
Then there exists a constant $c > 0$ such that, for infinitely many $N \in \NN$, we have
\[
q(P_N) \le \frac{c}{N^{1/d} (\log N)^{1/(d(d-1))}}.
\]
In particular, $\Scal$ is not quasi-uniform.
\end{theorem}

We require some preparation for the proof.

\begin{lemma}\label{lem:LTE}
Let $p,q \ge 2$ be integers such that $p \equiv 1 \pmod{q}$.
Then, for every $k\in\NN_0$, the number $p^{q^k} - 1$ is divisible by $q^{k+1}$.
\end{lemma}

\begin{proof}    
Let $v_k = p^{q^{k}}-1$.
We show the statement by induction on $k$. 
For $k=0$, the assumption $p \equiv 1 \pmod{q}$ gives $q \mid v_0 = p-1$.
For the induction step, assume $q^{k+1} \mid v_k$ and consider $v_{k+1}$.
Using the binomial expansion, we write
\[
v_{k+1} =  (v_k + 1)^q - 1
= \sum_{i=1}^q \binom{q}{i} v_k^{i}= qv_k+\binom{q}{2}v_k^2+\cdots+\binom{q}{q-1}v_k^{q-1}+v_k^{q}.
\]
By the induction hypothesis, $v_k$ is a multiple of $q^{k+1}$.
Since $q>1$, each term on the right-hand side is therefore a multiple of $q^{k+2}$.
Hence $q^{k+2} \mid v_{k+1}$, completing the induction.
\end{proof}

With this lemma in hand, we obtain the following proposition.

\begin{proposition}\label{prop:dim2}
Let $d \ge 2$, and let $b_1, \dots, b_d$ be pairwise relatively prime integers. We abbreviate Euler's totient function of $b_1$ by $r$, i.e., $\phi(b_1)=:r$. Let $\ell$, $k$, $c_2, \dots, c_d\in\NN$, and define 
\[ n=(b_1^{k}-1) b_1^{\ell+1},\quad m=n + (b_1+1) b_1^\ell \prod_{j=2}^d b_j^{rc_jb_1^k}.\]
Then the following bounds hold:
\[
|\varphi_{b_j}(n)-\varphi_{b_j}(m)|\leq \frac{1}{b_j^{rc_jb_1^k}}
\]
for all $2 \le j \le d$ and 
\[
|\varphi_{b_1}(n)-\varphi_{b_1}(m)|\leq \frac{2}{b_1^{\ell+k+1}}.
\]
\end{proposition}
\begin{proof}
For the first bound, observe that $m$ and $n$ share the first $rc_jb_1^k$ digits in their base $b_j$ expansions for any $2 \le j \le d$. Hence
\[|\varphi_{b_j}(n)-\varphi_{b_j}(m)|\leq \sum_{i=rc_jb_1^k+1}^\infty \frac{{b_j}-1}{b_j^{i}}=\frac{1}{b_j^{rc_jb_1^k}}.\]

For the second bound, we apply Lemma~\ref{lem:LTE} with $p= \prod_{j=2}^d b_j^{rc_j}$ and $q = b_1$, for which the assumption $p \equiv 1 \pmod{q}$ holds. The base $b_1$ expansions of $n$, the increment term, and $m$ are given, respectively, by 
\begin{align*}
    n & =(\underbrace{0,0,\ldots,0}_{\ell},0,\underbrace{{b_1}-1,{b_1}-1,\ldots,{b_1}-1}_{k})_{b_1},\\
    (b_1+1) b_1^\ell \prod_{j=2}^d b_j^{rc_jb_1^k} & =(\underbrace{0,0,\ldots,0}_{\ell},1,1,\underbrace{0,0,\ldots,0}_{k-1},\ldots)_{b_1}\\
    m & =(\underbrace{0,0,\ldots,0}_{\ell},1,\underbrace{0,0,\ldots,0}_{k},\ldots)_{b_1}.
\end{align*}
From these expansions, we obtain
\[
\varphi_{b_1}(n)= \frac{1}{b_1^{\ell+1}} - \frac{1}{b_1^{\ell+k+1}}
\]
and 
\[
\frac{1}{b_1^{\ell+1}} \le \varphi_{b_1}(m)
< \frac{1}{b_1^{\ell+1}} + \frac{1}{b_1^{\ell+k+1}}.
\]
Therefore,
\[ |\varphi_{b_1}(n)-\varphi_{b_1}(m)| \leq \frac{2}{b_1^{\ell+k+1}}.\]
This completes the proof.
\end{proof}

We aim to balance all upper bounds on $|\varphi_{b_j}(n)-\varphi_{b_j}(m)|$ for $1\le j\le d$. 
The following result suggests that this is possible.

\begin{proposition}\label{prop:existence_cj}
Let $d \ge 2$ and  $1<b_1 < \cdots < b_d$ be integers. For any integers $r,k \ge 1$, there exist integers {$1 \le c_2 \le b_1^{k(d-2)}$ and $0\leq c_j\leq c_2$ for $3 \le j \le d$} such that
\begin{equation}\label{eq:simul-approx}
b_d^{-r} b_2^{rc_2b_1^k} \le b_j^{rc_jb_1^k} \le b_d^r  b_2^{r c_2 b_1^k}
\end{equation}
holds for all $2 \le j \le d$.
\end{proposition}

\begin{proof}
For $d=2$ or $j=2$, the inequalities \eqref{eq:simul-approx} hold
trivially.
Hence, we assume $d \ge 3$.
Let $\alpha_j =  \log b_2 /\log b_j$ for $j=3,\dots,d$.
By Dirichlet's simultaneous approximation theorem \cite[Chapter~2]{Sch80}, for any $k \in \NN$, there exist integers $1 \le c_2 \le b_1^{k(d-2)}$ and $c_j \in \ZZ$ for $3 \le j \le d$ such that
\begin{equation}\label{eq:simul_eq1}
|c_2\alpha_j - c_j| \le \dfrac{1}{b_1^k}     
\end{equation}
for all $3 \le j \le d$.
Since $0<\alpha_j<1$, we may choose $0 \le c_j \le c_2 \le b_1^{k(d-2)}$.

It follows from \eqref{eq:simul_eq1} that
\[
|rb_1^kc_2\log b_2 - rb_1^kc_j \log b_j| \le r\log b_j \le r\log b_d,
\]
which is equivalent to \eqref{eq:simul-approx}.
\end{proof}

We are now ready to prove Theorem~\ref{thm:halton_not_quasi-uniform}.

\begin{proof}[Proof of Theorem~\ref{thm:halton_not_quasi-uniform}]
    We abbreviate Euler's totient function of $b_1$ by $r$, i.e. $\phi(b_1)=:r$. For any $k\in\NN$, let $(c_2, \dots,c_d) := (c_2(k), \dots, c_d(k))$ be integers satisfying \eqref{eq:simul-approx} with $c_j \le b_1^{k(d-2)}$, and define
    \[
    M := M(k) = b_d^rb_2^{rc_2b_1^k}, \quad
    N := N(k) = \left\lceil (b_1+2)\frac{M^{d}}{b_1^k}\right\rceil, \quad
    \ell := \ell(k) = \left\lfloor \log_{b_1} \frac{M}{b_1^k}\right\rfloor.
    \]
    Since $b_2^{r b_1^k}-1$ is a multiple of $b_1^{k+1}$ (see Lemma~\ref{lem:LTE}), we have $M \ge b_2^{rb_1^k} > b_1^{k+1}$ and hence $\ell \geq 1$. 
    Let $n=(b_1^{k}-1) b_1^{\ell+1}$ and $m=n + (b_1+1) b_1^\ell \prod_{j=2}^d b_j^{rc_jb_1^k}$.
    We show that the $n$-th and $m$-th points of the Halton sequence are sufficiently close.

    First, we show that $0\le n<m <N$ holds. Using Proposition~\ref{prop:existence_cj} together with the definitions of $n$, $M$, $\ell$, and $N$, we obtain the following chain of inequalities:
    \begin{align*}
    m & = n + (b_1+1) b_1^\ell \prod_{j=2}^d b_j^{rc_jb_1^k}\\
    &< b_1^\ell M +  (b_1+1) b_1^\ell M^{d-1} 
    \le (b_1+2) b_1^\ell M^{d-1}
    \le (b_1+2)\dfrac{M}{b_1^k}M^{d-1}
    \le N.
    \end{align*}

    Next, we derive the inequalities in \eqref{eq:M_lower_bound} that are needed for the proof.
    Choose $k$ sufficiently large so that $b_1^k \ge 2(b_1+2) b_d^{rd}$.
    Then we have
    \[
    N \le 2(b_1+2)\dfrac{M^d}{b_1^k}
    \le b_2^{rc_2b_1^kd} \le b_2^{rd b_1^{k(d-1)}},
    \]
    where we used $c_2\leq b_1^{k(d-2)}$, and thus
    \[
    \left(\dfrac{\log N}{rd \log b_2}\right)^{1/(d-1)} \le b_1^k.
    \]
    This implies that
    \begin{equation} \label{eq:M_lower_bound}
    M^d \ge \frac{1}{2(b_1+2)} N b_1^k
    \ge C N(\log N)^{1/(d-1)},
    \end{equation}
    with constant $C = 2^{-1}(b_1+2)^{-1}(rd \log b_2)^{-1/(d-1)}$, depending only on $b_1, b_2$ and $d$.

    Finally, we bound the distance between the points $\bsx_n$ and $\bsx_m$.
    From Proposition~\ref{prop:dim2}, the definition of $M$, and \eqref{eq:simul-approx}, for $2 \le j \le d$, we derive
    \[
    |\varphi_{b_j}(n)-\varphi_{b_j}(m)| \le \frac{1}{b_j^{rc_jb_1^k}}
    \le \frac{b_d^r}{b_2^{rc_2b_1^k}} = \dfrac{b_d^{2r}}{M}.
    \]
    Proposition~\ref{prop:dim2}, together with the definition of $\ell$, implies that
    \[
    |\varphi_{b_1}(n)-\varphi_{b_1}(m)|
    \le \frac{2}{b_1^{\ell+k+1}}
    \le \frac{2}{b_1^{\log_{b_1}(M/b_1^k)+k}} =
    \frac{2}{M}.
    \]
    Hence, using \eqref{eq:M_lower_bound}, which holds true under the assumption that $k$ is sufficiently large so that $b_1^k \ge 2(b_1+2) b_d^{dr}$, we have
    \[
    \|\bsx_n-\bsx_m\|= \left(\sum_{j=1}^{d}|\varphi_{b_j}(n)-\varphi_{b_j}(m)|^2\right)^{1/2}
    \le \dfrac{d^{1/2}b_d^{2r}}{M}
    \le \dfrac{d^{1/2}b_d^{2r}}{C^{1/d}}\dfrac{1}{N^{1/d}(\log N)^{1/(d(d-1))}}.
    \]

    Thus, there exist infinitely many $N$ for which this inequality holds. In particular, this proves that the Halton sequence is not quasi-uniform.
\end{proof}

\begin{op}
We have shown that there exist infinitely many $N$ for which the separation radius satisfies
\[
q(P_N) \lesssim \frac{1}{N^{1/d} (\log N)^{1/(d(d-1))}}.
\]
It remains an open question to determine the exact order of the separation radius for the Halton sequence.
\end{op}

\section{Further results: Halton-type sequences}\label{sec:Halton-type}

As mentioned in Section~\ref{sec:intro}, for any prime base $p$, the $p$-dimensional Faure sequence fails to be quasi-uniform; see \cite[Theorem~4.3]{DGS25}. The proof in \cite{DGS25} relies on properties of Pascal matrices modulo~$p$ together with Lucas's theorem. In this section, we work out an alternative proof of \cite[Theorem~4.3]{DGS25} from the viewpoint of Halton-type sequences, using arithmetic in the ring of polynomials. We also present further examples of Halton-type sequences that are not quasi-uniform.

\subsection{Definitions}
To begin with, we introduce a polynomial analogue of the radical inverse function $\varphi_b$ and recall the definition of Halton-type sequences in the sense of Hofer \cite{Ho13}. For more general definitions of Halton-type sequences, we refer to \cite{Ho18}.
Note that the notion of Halton-type sequences is different from that of polynomial Halton sequences defined by Tezuka \cite{Te93}.
In what follows, $\PP$ denotes the set of prime numbers. For $p\in \PP$, let $\FF_p$ be the finite field with $p$ elements, identified with the set $\{0,1,\ldots,p-1\}$ equipped with addition and multiplication modulo $p$. Moreover, we write $\FF_p[X]$ for the ring of polynomials over $\FF_p$.
   
\begin{definition}[$b(X)$-adic radical inverse function $\varphi_{b(X)}$]\label{def:polynomial_radical_inverse}
	Let $p \in \PP$, and let $b(X)\in\FF_p[X]$ be a polynomial of degree $\deg(b(X))=:e\geq 1$, $e\in\NN$. For $n\in \NN_0$ with its $p$-adic expansion $n=n_0+n_1p+n_2p^2+\cdots$, we associate the polynomial $n(X)=n_0X^0+n_1X^1+n_2X^2+\cdots \in\FF_p[X]$. We then expand $n(X)$ in base $b(X)$, that is 
	\[ n(X)=a_0(X)b(X)^0+a_1(X)b(X)^1+a_2(X)b(X)^2+\cdots , \]
	with $a_j(X)\in\FF_p[X]$ satisfying $\deg(a_j(X))<e$ for every $j\geq 0$. 
	The $b(X)$-adic radical inverse function $\varphi_{b(X)}$ is defined by 
	\[ \varphi_{b(X)}(n):=\sum_{j\geq 0}\frac{a_j(p)}{(p^e)^{j+1}},\]
    {where $a_j(p)$ is computed by identifying the coefficients of $a_j$ with the integers in $\{0,1,\dots,p-1\}$}.
	Since $0\leq a_j(p)<p^e$, we have $\varphi_{b(X)}:\NN_0 \to [0,1)$. 
\end{definition}
	
\begin{definition}[Halton-type sequence]
	Let $p\in\PP$ and $d\in\NN$. Let $b_1(X),\ldots,b_d(X)\in \FF_p[X]$ be pairwise relatively prime polynomials with degrees $e_1,\ldots ,e_d\geq 1$. Then the sequence $\{\bsx_n\}_{n\geq 0}$ defined by
	$$\bsx_n:=(\varphi_{b_1(X)}(n),\ldots,\varphi_{b_d(X)}(n))$$
	is called the \emph{Halton-type sequence} in bases $b_1(X),\ldots,b_d(X)$.
\end{definition}

Additionally, we introduce the digital construction and $(t,d)$-sequences in a reduced form, which is sufficient for our purposes (for more details, see \cite[Chapter~4]{Nie92}). 

\begin{definition}[Digital construction and $(t,d)$-sequence]
	Let $p\in\PP$, $d\in\NN$, and let $C_1,\ldots,C_d\in\FF_p^{\NN\times \NN_0}$. A \emph{digital sequence} $\{\bsx_n\}_{n\geq 0}$ with $\bsx_n=(x_{n,1},\ldots,x_{n,d})\in[0,1]^d$ is constructed as follows. For each $\ell \in\{1,\ldots,d\}$ and $n\in\NN_0$, write the base $p$ expansion $n=n_0+n_1p+n_2p^2+\cdots$, and set $\vec{n}:=(n_0,n_1,n_2,\ldots)^{\top}\in \FF_p^{\NN_0}$. Compute $C_{\ell}\cdot\vec{n}=:(y^{(\ell)}_{n,1},y^{(\ell)}_{n,2},y^{(\ell)}_{n,3},\ldots)^{\top}\in\FF_p^{\NN}$, and define 
    \[ x_{n,\ell} := \sum_{j=1}^\infty\frac{y^{(\ell)}_{n,j}}{p^{j}}.\]
    
    Let $t\in\NN_0$. The sequence $\{\bsx_n\}_{n\geq 0}$ is called a \emph{$(t,d)$-sequence over $\FF_p$} if for every integer $m>t$ the following holds: for any choice of non-negative integers $w_1,\ldots,w_d$ with $w_1+\cdots+w_d=m-t$, the set of row vectors formed by taking, for each $\ell=1,\ldots,d$, the first $w_{\ell}$ rows of the matrix consisting of the first $m$ columns of $C_{\ell}$ is linearly independent over $\FF_p$.
\end{definition}

It is important to note that the Faure sequence can be viewed both as a digital sequence and as a Halton-type sequence as follows.

\begin{remark}
    Define the Pascal matrix $P$ as $P=(\binom{j}{i})_{i,j\geq 0}$, where we set $\binom{j}{i}=0$ if $j<i$. Let $I$ be the $\infty\times\infty$ identity matrix. For a prime $p\in\PP$, set $C_1:=I$, $C_2:=P\pmod{p}$, $\ldots$, $C_p:=P^{p-1}\pmod{p}$. Then the $p$-dimensional Faure sequence in base $p$ is the digital sequence generated by $C_1,\ldots,C_p$, which is also a $(0,p)$-sequence over $\FF_p$.

    Similarly, let $b_1(X)=X$, $b_2(X)=X-1$, $\ldots$, $b_p(X)=X-(p-1)$ over $\FF_p$. Then the $p$-dimensional Faure sequence in base $p$ is the Halton-type sequence in these bases (see, e.g., \cite[Example~4]{Ho13}). Moreover, in the case $p=2$, the corresponding Halton-type sequence coincides with the two-dimensional Sobol' sequence.
\end{remark}

Before moving on, we point out that, as proven in \cite[Proof of Theorem~4]{Ho18}, any Halton-type sequence is a $(0,\boldsymbol{e},d)$-sequence in base $p$ with $\boldsymbol{e}=(e_1,\ldots,e_d)$ in the sense of Tezuka \cite{Te13}, who introduced the concept of $(t,\boldsymbol{e},d)$-sequences.

\subsection{Results}

First, we prove that any Halton-type sequence achieves the optimal rate for the covering radius. 
\begin{proposition}
    Let $p\in\PP$ and $d\in\NN$. Let $b_1(X),\ldots,b_d(X)\in \FF_p[X]$ be pairwise relatively prime polynomials with degrees $e_1,\ldots ,e_d\geq 1$. Let $\Scal$ be the $d$-dimensional Halton-type sequence in bases $b_1(X),\ldots,b_d(X)$. Then, for any $N\ge 1$, we have
    \[ h(P_N) \le \frac{\sqrt{d}}{N^{1/d}} \max_{1 \le \ell \le d} p^{e_{\ell}}. \]
\end{proposition}

\begin{proof}
    The argument follows the proof of \cite[Theorem~6.12]{Nie92}. For each $\ell\in \{1,\ldots,d\}$, let $j_{\ell}$ be the unique non-negative integer such that $p^{e_{\ell}j_{\ell}}\le N^{1/d}<p^{e_{\ell}(j_{\ell}+1)}$. Consider the partition of $[0,1)^d$ into intervals of the form
    \[ J=\prod_{\ell=1}^{d}\left[ \frac{c_{\ell}}{p^{e_{\ell}j_{\ell}}}, \frac{c_{\ell}+1}{p^{e_{\ell}j_{\ell}}}\right)\]
    with integers $0\le c_{\ell}<p^{e_{\ell}j_{\ell}}$ for $1\le \ell\le d$. As shown in \cite[Proof of Theorem~4]{Ho18}, among the first $N'=\prod_{\ell=1}^d p^{e_{\ell}j_{\ell}}$ points of $\Scal$, exactly one point lies in each such interval $J$. Hence
    \[ h(P_N)\le h(P_{N'})\le \operatorname{diam}(J)\le \sqrt{d}\max_{1 \le \ell \le d}\frac{1}{p^{e_{\ell}j_{\ell}}}\le \frac{\sqrt{d}}{N^{1/d}} \max_{1 \le \ell \le d} p^{e_{\ell}},\]
    which proves the claim.
\end{proof}

Now the situation is analogous to that of the classical Halton sequence: the question of whether a Halton-type sequence is quasi-uniform reduces to determining whether its separation radius decays no faster than $N^{-1/d}$. At present, in contrast to the classical Halton sequences, we do not know whether every Halton-type sequence is not quasi-uniform for $d\ge 2$ (cf.\ Open Problem \ref{op2}). In what follows, we present some examples whose separation radius decays faster than $N^{-1/d}$, and thus they are not quasi-uniform. These examples are collected in the following theorem.

\begin{theorem}\label{thm:Halton-type}
    Let $p\in\PP$ and $d\in\NN$. Let $b_1(X),\ldots,b_d(X)\in \FF_p[X]$ be pairwise relatively prime polynomials with degrees $e_1,\ldots ,e_d\geq 1$. Let $\Scal$ be the $d$-dimensional Halton-type sequence in bases $b_1(X),\ldots,b_d(X)$. Consider the following cases:
    \begin{enumerate}
        \item $d=2$, $b_1(X)=X+a$, and $b_2(X)=X+a+1$ with $a\in \FF_p$.
        \item $d=3$, $p=2$, $b_1(X)=X$, $b_2(X)=X+1$, and $b_3(X)=X^2+X+1$.
        \item $d=p$, $b_1(X)=X$, $b_2(X)=X-1$, $\ldots$, $b_p(X)=X-(p-1)$.
    \end{enumerate}
    Then there exists a constant $c>0$, depending on the case, such that for infinitely many $N\in\NN$ we have
    \[
      q(P_N) \le \frac{c}{N^{1/(d-1)}}.
    \]
    In particular, each of these sequences is not quasi-uniform.
\end{theorem}

\begin{remark}
    The first case shows that any two-dimensional projection of the successive coordinates of the Faure sequence is not quasi-uniform. In particular, when $p=2$, the cases $a=0$ and $a=1$ coincide, and the corresponding sequence is the two-dimensional Sobol' sequence that is already known to be not quasi-uniform \cite{G24,S25}. Thus, this result can be seen as a generalization of the known result. 

    The second case is related to the three-dimensional Sobol' sequence, as both sequences are $(0,(1,1,2),3)$-sequences constructed from the same polynomials $X,X+1,X^2+X+1$ over $\FF_2$. Hence, the result can be seen as an extension of \cite{G24,S25}.

    The last case corresponds exactly to the $p$-dimensional Faure sequence in base $p$. Since this sequence is already known to be not quasi-uniform \cite{DGS25}, our contribution here is to provide an alternative proof from the viewpoint of Halton-type sequences.
\end{remark}

To prove Theorem~\ref{thm:Halton-type}, we begin with our key lemma.
\begin{lemma}\label{lem:key}
	Let $p\in\PP$ and $\{\bsx_n\}_{n\geq 0}$ be the one-dimensional Halton-type sequence in base $b(X)$ of degree $e\geq 1$. Let $n,m\in\NN_0$ with associated polynomials $n(X)$ and $m(X)$ in $\FF_p[X]$. If there exists an integer $\ell\ge 1$ such that $b(X)^{\ell}|(n(X)-m(X))$, then $|\bsx_n-\bsx_m|\le p^{- e \ell}$. 
\end{lemma}

\begin{proof}
Consider the base-$b(X)$ expansions of $n(X)$ and $m(X)$;
\begin{align*}
    n(X) & = a_0(X)b(X)^0+a_1(X)b(X)^1+a_2(X)b(X)^2+\cdots \\
    m(X) & = c_0(X)b(X)^0+c_1(X)b(X)^1+c_2(X)b(X)^2+\cdots.
\end{align*}
It is straightforward to see that if $b(X)^{\ell}|(n(X)-m(X))$, then
\[ a_0(X)=c_0(X), \ldots, a_{\ell-1}(X)=c_{\ell-1}(X).\]
Hence, by Definition~\ref{def:polynomial_radical_inverse}, we have
\[ |\bsx_n-\bsx_m|\le \sum_{j\geq \ell}\frac{p^e-1}{(p^e)^{j+1}}=\frac{1}{p^{e\ell}},\]
which proves the lemma.
\end{proof}
	
\begin{proof}[Proof of Theorem~\ref{thm:Halton-type}]
    We prove the three cases separately. In what follows, we use the identity
    \[ (a(X)+b(X))^{p^w}=a(X)^{p^w}+b(X)^{p^w}\]
    for $a(X),b(X)\in \FF_p[X]$ and $w\in \NN$, often and without further comment.

    \textbf{First case.} For $w\in \NN$, set $N=p^m$ with $m=p^w$. Consider $0\le n_1,n_2<N$ such that
    \[ n_1(X)=1,\quad n_2(X)=(p-1)\sum_{j=1}^{m-1}b_2(X)^j. \]
    Then we have
    \[n_2(X)-n_1(X)=(p-1)\frac{b_2(X)^{m}-1}{b_2(X)-1}=(p-1)\frac{(b_2(X)-1)^m}{b_2(X)-1}=(p-1)b_1(X)^{m-1}.\]
    
    Applying Lemma~\ref{lem:key} with $b(X)=b_1(X)$ and $\ell=m-1$, we obtain
    \[  |\varphi_{b_1(X)}(n_1)-\varphi_{b_1(X)}(n_2)|\le \frac{1}{p^{m-1}}=\frac{p}{N}. \] 
    Moreover, a direct calculation gives
    \[ |\varphi_{b_2(X)}(n_1)-\varphi_{b_2(X)}(n_2)|=\left|\frac{1}{p}-(p-1)\sum_{j=1}^{m-1}\frac{1}{p^{j+1}}\right|=\frac{1}{p^{m}}=\frac{1}{N}.\]
    Thus we have
    \[ q(P_N)\le \frac{1}{2}\|\bsx_{n_1}-\bsx_{n_2}\|=\frac{\sqrt{1+p^2}}{2N}. \]
    Since this bound holds for any $w\in \NN$ and the corresponding $N$, the claim follows. 

    \textbf{Second case.} Similar to the first case, for $w\in \NN$, set $N=2^{2m}$ with $m=2^w$. Consider $0\le n_1,n_2<N$ such that 
    \[ n_1(X)=b_1(X), \quad  n_2(X)=1+ b_2(X)\sum_{j=1}^{m-1}b_3(X)^j. \]
    We observe that
    \[
          \sum_{j=0}^{m-1}b_3(X)^j =\frac{b_3(X)^{m}-1}{b_3(X)-1} = \frac{(b_3(X)-1)^m}{b_3(X)-1}=b_1(X)^{m-1}b_2(X)^{m-1}.
    \]
    It follows that 
    \[ n_2(X)-n_1(X)=1+b_2(X)\left( b_1(X)^{m-1}b_2(X)^{m-1}-1\right)-b_1(X)=b_1(X)^{m-1}b_2(X)^{m}.\]
    Hence, we have $b_1(X)^{m-1}\mid n_2(X)-n_1(X)$ and $b_2(X)^{m}\mid n_2(X)-n_1(X)$.

    Applying Lemma~\ref{lem:key}, we obtain
    \[  |\varphi_{b_1(X)}(n_1)-\varphi_{b_1(X)}(n_2)|\le \frac{1}{2^{m-1}}=\frac{2}{N^{1/2}}. \]
    and 
    \[  |\varphi_{b_2(X)}(n_1)-\varphi_{b_2(X)}(n_2)|\le \frac{1}{2^{m}}=\frac{1}{N^{1/2}}. \]
    Moreover, a direct calculation gives $\varphi_{b_3(X)}(n_1)=1/2$ and 
    \[ \varphi_{b_3(X)}(n_2)=\frac{1}{4}+\sum_{j=1}^{m-1}\frac{3}{4^{j+1}}=\frac{1}{2}-\frac{1}{2^{2m}}.\]
    Hence, $|\varphi_{b_3(X)}(n_1)-\varphi_{b_3(X)}(n_2)|=1/2^{2m}=1/N\le 1/N^{1/2}$.
    Thus we have
    \[ q(P_N)\le \frac{1}{2}\|\bsx_{n_1}-\bsx_{n_2}\|\le \frac{\sqrt{6}}{2N^{1/2}}. \]
    Since this bound holds for any $w\in \NN$ and the corresponding $N$, the claim follows. 
    
    \textbf{Third case.} Finally, we prove the last case. For $w\in \NN$, set $N=p^{m}$ with $m=(p-1)p^w$. Consider $0\le n_1,n_2<N$ such that 
    \[ n_1(X)=1, \quad  n_2(X)=(p-1)\sum_{j=1}^{m-1}X^j. \]
    Then, working in $\FF_p[X]$, we have
    \[ n_2(X)-n_1(X)=-1+(p-1)\sum_{j=1}^{m-1}X^j=(p-1)\frac{X^m-1}{X-1}=(p-1)\frac{X^m-1}{b_2(X)}.\]
    
    Since all non-zero elements of $\FF_p$ are roots of $X^{p-1}-1$, we have
    \[ \prod_{\ell=2}^{p}b_{\ell}(X)=X^{p-1}-1. \]
    By taking the $p^w$-th power of both sides, we get
    \[ \prod_{\ell=2}^{p}b_{\ell}(X)^{p^w}=(X^{p-1}-1)^{p^w}=X^{(p-1)p^w}-1=X^m-1. \]
    Hence $b_2(X)^{p^w-1}\mid n_2(X)-n_1(X)$ and $b_{\ell}^{p^w}(X)\mid n_2(X)-n_1(X)$ for all $\ell\in \{3,\ldots,p\}$.

    Applying Lemma~\ref{lem:key}, we obtain
    \[  |\varphi_{b_2(X)}(n_1)-\varphi_{b_2(X)}(n_2)|\le \frac{1}{p^{p^w-1}}=\frac{p}{N^{1/(p-1)}}. \]
    and 
    \[  |\varphi_{b_{\ell}(X)}(n_1)-\varphi_{b_{\ell}(X)}(n_2)|\le \frac{1}{p^{p^w}}=\frac{1}{N^{1/(p-1)}}, \]
    for all $\ell\in \{3,\ldots,p\}$.
    Moreover, a direct calculation gives $\varphi_{b_1(X)}(n_1)=1/p$ and 
    \[ \varphi_{b_1(X)}(n_2)=\sum_{j=1}^{m-1}\frac{p-1}{p^{j+1}}=\frac{1}{p}-\frac{1}{p^{m}}.\]
    Hence, $|\varphi_{b_1(X)}(n_1)-\varphi_{b_1(X)}(n_2)|=1/p^{m}=1/N\le 1/N^{1/(p-1)}$.
    Thus we have
    \[ q(P_N)\le \frac{1}{2}\|\bsx_{n_1}-\bsx_{n_2}\|\le \frac{\sqrt{p^2+p-1}}{2N^{1/(p-1)}}. \]
    Since this bound holds for any $w\in \NN$ and the corresponding $N$, the claim follows. 
\end{proof}

\begin{remark}
	The proof of the third case relies on the equality $d=p$; that is, our argument works only when all components of the $p$-dimensional Faure sequence in base $p$ are considered. This is also the case in \cite{DGS25}.
        
    Furthermore, our bound on $|\varphi_{b_{\ell}(X)}(n_1)-\varphi_{b_{\ell}(X)}(n_2)|$ for $\ell\in \{2,\ldots,p\}$ relies on the fact that $\bsx_{n_1}$ and $\bsx_{n_2}$ share the first $p^w-1$ or $p^w$ digits in their base $p$ expansions. From our proof, we can easily derive a generalization of \cite[Theorem~4.2]{DGS25}: let $L_2,L_3,\ldots,L_{p-1}$ be any non-singular lower triangular matrices over $\FF_p$, then the $(0,p)$-sequence constructed by $(I,L_2P,\ldots,L_pP^{p-1})\pmod{p}$ satisfies 
		\[ q(P_N)\lesssim \frac{1}{N^{1/(p-1)}},\]
	for $N=p^m$ with $m=(p-1)p^w$, $w\in\NN$. Note the first component is not changed. For the other components, observe if $x_{n_1,\ell}$ and $x_{n_2,\ell}$ of a digital sequence constructed by $C_{\ell}$ share the first $w$ digits in the base $p$ expansions, then the corresponding elements of a digital sequence constructed by $L_{\ell}C_{\ell}$ also share the first $w$ digits.
\end{remark}

In this section, we have considered three examples of multi-dimensional Halton-type sequences over $\FF_p$. The observation is that, for each example, there exist suitable pairs of indices such that the corresponding points are too close to each other, which disproves quasi-uniformity. One natural question is whether a general formula exists for these indices in terms of the base polynomials. We state this as an open problem.

\begin{op}\label{op2}
Prove or disprove that no Halton-type sequence is quasi-uniform for $d\ge 2$. If this is the case, find pairs of indices $n_1$ and $n_2$ for an arbitrary Halton-type sequence in given bases $b_1(X),\ldots,b_d(X)$ which disproves quasi-uniformity.
\end{op}

\bibliographystyle{amsplain} 
\bibliography{ref}

\end{document}